\documentclass{article}%
\usepackage{amssymb}
\usepackage{amsmath}
\usepackage{amsfonts}
\usepackage{graphicx}%
\providecommand{\U}[1]{\protect\rule{.1in}{.1in}}
\newtheorem{theorem}{Theorem}[section]

\newtheorem{corollary}[theorem]{Corollary}

\newtheorem{example}[theorem]{Example}

\newtheorem{lemma}[theorem]{Lemma}

\newtheorem{problem}[theorem]{Problem}

\newenvironment{proof}[1][Proof]{\noindent\textbf{#1.} }{\ \rule{0.5em}{0.5em}}
\begin{document}

\author{Vadim E. Levit\\Department of Mathematics\\Ariel University, Israel\\levitv@ariel.ac.il
\and Eugen Mandrescu\\Department of Computer Science\\Holon Institute of Technology, Israel\\eugen\_m@hit.ac.il}
\title{On corona of K\"{o}nig-Egerv\'{a}ry graphs}
\date{}
\maketitle

\begin{abstract}
Let $\alpha(G)$ denote the cardinality of a maximum independent set and
$\mu(G)$ be the size of a maximum matching of a graph $G=\left(  V,E\right)
$. If $\alpha(G)+\mu(G)=\left\vert V\right\vert $, then $G$ is a
\textit{K\"{o}nig-Egerv\'{a}ry graph}, and $G$ is a $1$%
-\textit{K\"{o}nig-Egerv\'{a}ry graph} whenever $\alpha(G)+\mu(G)=\left\vert
V\right\vert -1$. The \textit{corona} $H\circ\mathcal{X}$ of a graph $H$ and a
family of graphs $\mathcal{X}=\left\{  X_{i}:1\leq i\leq\left\vert
V(H)\right\vert \right\}  $ is obtained by joining each vertex $v_{i}$ of $H$
to all the vertices of the corresponding graph $X_{i},i=1,2,...,\left\vert
V(H)\right\vert $.

In this paper we completely characterize graphs whose coronas are
$k$-K\"{o}nig-Egerv\'{a}ry graphs, where $k\in\left\{  0,1\right\}  $.


\end{abstract}

\section{Introduction}

Throughout this paper $G=(V,E)$ is a finite, undirected, loopless graph
without multiple edges, with vertex set $V=V(G)$ of cardinality $\left\vert
V\left(  G\right)  \right\vert =n\left(  G\right)  $, and edge set $E=E(G)$ of
size $\left\vert E\left(  G\right)  \right\vert =m\left(  G\right)  $.

If $X\subset V$, then $G[X]$ is the subgraph of $G$ induced by $X$. By $G-v$
we mean the subgraph $G[V-\left\{  v\right\}  ]$, for $v\in V$. The
\textit{neighborhood} of a vertex $v\in V$ is the set $N(v)=\{w:w\in V$ and
$vw\in E\}$. The \textit{neighborhood} of $A\subseteq V$ is $N(A)=\{v\in
V:N(v)\cap A\neq\emptyset\}$.

A set $S\subseteq V(G)$ is \textit{independent} if no two vertices belonging
to $S$ are adjacent. The \textit{independence number} $\alpha(G)$ is the size
of a largest independent set (i.e., of a \textit{maximum independent set}) of
$G$. Let $\mathrm{core}(G)$ be the intersection of all maximum independent
sets of $G$.

A \textit{matching} in a graph $G$ is a set of edges $M\subseteq E\left(
G\right)  $ such that no two edges of $M$ share a common vertex. A matching of
maximum cardinality $\mu(G)$ is a \textit{maximum matching}, and a
\textit{perfect matching} is one saturating all vertices of $G$; a matching is
\textit{almost perfect} if only a single vertex is left unmatched. Given a
matching $M$ in $G$, a vertex $v\in V\left(  G\right)  $ is $M$%
-\textit{saturated} if there exists an edge $e\in M$ incident with $v$. If an
edge $uv$ $\in$ $E\left(  G\right)  $ belongs to a maximum matching, then $uv$
is \textit{matching covered}. If every edge of $G$ is matching covered, then
$G$ is \textit{matching covered}.

It is known that
\[
\alpha(G)+\mu(G)\leq n\left(  G\right)  \leq\alpha(G)+2\mu(G)
\]
hold for every graph $G$ \cite{BGL2002}.

The \textit{K\"{o}nig deficiency} of graph $G$ is $\kappa\left(  G\right)
=n\left(  G\right)  -\left(  \alpha(G)+\mu(G)\right)  $ \cite{BartaKres2020}.
Thus, a graph $G$ is $k$-\textit{K\"{o}nig-Egerv\'{a}ry} if and only if
$\kappa\left(  G\right)  =k$ \cite{LevMan2023,LevMan2024b}. In particular, if
$k=0$, then $G$ is a \textit{K\"{o}nig-Egerv\'{a}ry graph }%
\cite{dem,gavril1977testing,ster}. K\"{o}nig-Egerv\'{a}ry graphs have been
extensively studied \cite{Larson2007,
Larson2011,LevMan2002a,LevMan2003,LevMan2006,LevMan2012a, LevMan2013b,
LevMan2022,lovasz1983ear}. Some structural properties of $1$%
-K\"{o}nig-Egerv\'{a}ry graphs may be found in \cite{LevMan2024b}. For
instance, both $G_{1}$ and $G_{2}$ from Figure \ref{fig123} are $1$%
-K\"{o}nig-Egerv\'{a}ry graphs. \begin{figure}[h]
\setlength{\unitlength}{1cm}\begin{picture}(5,1.2)\thicklines
\multiput(2,0)(1,0){4}{\circle*{0.29}}
\multiput(2,1)(1,0){3}{\circle*{0.29}}
\put(2,0){\line(1,0){3}}
\put(2,0){\line(0,1){1}}
\put(2,0){\line(1,1){1}}
\put(2,0){\line(2,1){2}}
\put(2,1){\line(1,-1){1}}
\put(2,1){\line(2,-1){2}}
\put(2,1){\line(3,-1){3}}
\put(3,0){\line(0,1){1}}
\put(3,0){\line(1,1){1}}
\put(3,1){\line(1,-1){1}}
\put(3,1){\line(2,-1){2}}
\put(4,0){\line(0,1){1}}
\put(4,1){\line(1,-1){1}}
\put(5.4,0){\makebox(0,0){$v_{1}$}}
\put(4.4,1){\makebox(0,0){$v_{2}$}}
\put(1,0.5){\makebox(0,0){$G_{1}$}}
\multiput(8,0)(1,0){5}{\circle*{0.29}}
\multiput(11,1)(1,0){2}{\circle*{0.29}}
\put(8,1){\circle*{0.29}}
\put(8,0){\line(1,0){4}}
\put(8,0){\line(0,1){1}}
\put(8,1){\line(1,-1){1}}
\put(11,1){\line(1,0){1}}
\put(10,0){\line(1,1){1}}
\put(12,0){\line(0,1){1}}
\put(9.5,0.2){\makebox(0,0){$e_{1}$}}
\put(10.55,0.2){\makebox(0,0){$e_{2}$}}
\put(7,0.5){\makebox(0,0){$G_{2}$}}
\end{picture}\caption{$G_{1}-v_{1}$, $G_{2}-e_{2}$ are K\"{o}nig-Egerv\'{a}ry
graphs, while $G_{1}-v_{2}$ and $G_{2}-e_{1}$ are not K\"{o}nig-Egerv\'{a}ry
graphs.}%
\label{fig123}%
\end{figure}
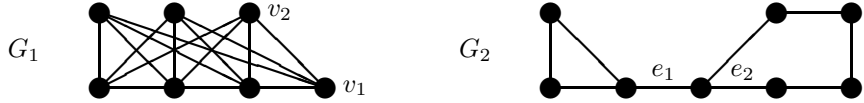

A subfamily of $1$-K\"{o}nig-Egerv\'{a}ry graphs was characterized in
\cite{DZ2021}, settling an open problem posed in \cite{LevMan2019}.

Let $H$ be a graph with $V(H)=\{v_{i}:1\leq i\leq n\left(  H\right)  \}$, and
$\mathcal{X}=\left\{  X_{i}:1\leq i\leq n\left(  H\right)  \right\}  $ be a
family of graphs, where $X_{i}=(V_{i},E_{i})=(\{u_{ij}:1\leq j\leq
q_{i}\},E_{i})$. Joining each $v_{i}\in V(H)$ to all the vertices of $X_{i}$,
we obtain a new graph, the \textit{corona} of $H$ and $\mathcal{X}$, which we
denote by $G=H\circ\mathcal{X}=H\circ\left\{  X_{i}:1\leq i\leq n\left(
H\right)  \right\}  $ \cite{FruchtHarary}. More precisely, the graph
$G=H\circ\mathcal{X}$ has
\[
V(G)=V(H)\cup V_{1}\cup\cdots\cup V_{n\left(  H\right)  }%
\]
as a vertex set and
\[
E(G)=E(H)\cup E_{1}\cup\cdots\cup E_{n\left(  H\right)  }\cup\{v_{1}%
u_{1j}:1\leq j\leq q_{1}\}\cup\cdots\cup\{v_{n}u_{nj}:1\leq j\leq q_{n}\}
\]
as an edge set. It is easy to see that $H\circ\mathcal{X}$ is connected if and
only if $H$ is connected. If $X_{i}=X,1\leq i\leq n$, we write $H\circ X$, and
in this case, $G$ is the \textit{corona} of $H$ and $X$. If all $X_{i}$ are
complete graphs, then $H\circ\mathcal{X}$ is called the \textit{clique corona}
of $H$ and $\mathcal{X}$ (see Figure \ref{fig12} for an example, where
$H=K_{3}+v_{3}v_{4}$). 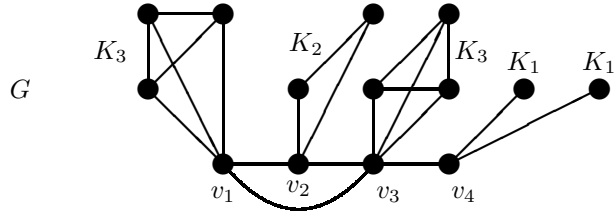
\begin{figure}[h]
\setlength{\unitlength}{1cm}\begin{picture}(5,2.75)\thicklines
\multiput(6,0.5)(1,0){4}{\circle*{0.29}}
\put(6,2.5){\circle*{0.29}}
\multiput(5,1.5)(0,1){2}{\circle*{0.29}}
\multiput(7,1.5)(1,0){5}{\circle*{0.29}}
\multiput(8,2.5)(1,0){2}{\circle*{0.29}}
\multiput(6,0.5)(1,0){3}{\line(1,0){1}}
\multiput(7,0.5)(1,0){2}{\line(0,1){1}}
\multiput(7,0.5)(1,0){2}{\line(1,2){1}}
\multiput(7,1.5)(1,0){2}{\line(1,1){1}}
\multiput(8,0.5)(1,0){2}{\line(1,1){1}}
\put(5,1.5){\line(1,1){1}}
\put(5,1.5){\line(0,1){1}}
\put(5,2.5){\line(1,0){1}}
\put(5,1.5){\line(1,-1){1}}
\put(5,2.5){\line(1,-2){1}}
\put(6,0.5){\line(0,1){2}}
\put(8,1.5){\line(1,0){1}}
\put(9,1.5){\line(0,1){1}}
\put(9,0.5){\line(2,1){2}}
\qbezier(6,0.5)(7,-0.7)(8,0.5)
\put(10,1.85){\makebox(0,0){$K_{1}$}}
\put(11,1.85){\makebox(0,0){$K_{1}$}}
\put(6,0.1){\makebox(0,0){$v_1$}}
\put(7,0.17){\makebox(0,0){$v_2$}}
\put(8.2,0.1){\makebox(0,0){$v_3$}}
\put(9.2,0.1){\makebox(0,0){$v_4$}}
\put(4.5,2){\makebox(0,0){$K_3$}}
\put(7.1,2.1){\makebox(0,0){$K_2$}}
\put(9.3,2){\makebox(0,0){$K_3$}}
\put(3.3,1.5){\makebox(0,0){$G$}}
\end{picture}\caption{$G=(K_{3}+v_{1}v_{3})\circ\{K_{3},K_{2},K_{3},2K_{1}\}$
is not a K\"{o}nig-Egerv\'{a}ry graph.}%
\label{fig12}%
\end{figure}

If $K_{1}=(\{v\},\emptyset)$, then we use $v\circ X$ instead of $K_{1}\circ X$.

In this paper we characterize graphs whose coronas are $k$%
-K\"{o}nig-Egerv\'{a}ry graphs for $k\in\left\{  0,1\right\}  $.

\section{Results}

\begin{theorem}
\label{THeorem1}Let $G=H\circ\left\{  X_{i}:1\leq i\leq n\left(  H\right)
\right\}  $ and $F=\left\{  v_{i}\in V\left(  H\right)  :\mu(X_{i})=\mu
(v_{i}\circ X_{i})\right\}  $.

\textrm{(i) }$\alpha(G)=%
{\displaystyle\sum\limits_{i=1}^{n\left(  H\right)  }}
\alpha(X_{i})$;

\textrm{(ii) }$\mu(G)=\mu\left(  H\left[  F\right]  \right)  +%
{\displaystyle\sum\limits_{i=1}^{n\left(  H\right)  }}
\mu(v_{i}\circ X_{i})$;

\textrm{(iii) }$%
{\displaystyle\sum\limits_{i=1}^{n\left(  H\right)  }}
\mu(v_{i}\circ X_{i})=%
{\displaystyle\sum\limits_{i=1}^{n\left(  H\right)  }}
\mu(X_{i})+\left\vert V\left(  H\right)  -F\right\vert $.
\end{theorem}

\begin{proof}
\textit{(i)}\textrm{ }The equality holds, since $\alpha(X_{i})=\alpha
(v_{i}\circ X_{i})$ for every $1\leq i\leq n\left(  H\right)  $.

\textit{(ii)} Let $M$ be a maximum matching of $G$. Suppose $v_{i}v_{j}\in
M\cap E\left(  H\right)  $ such that $v_{i}\notin F$. Then $M_{1}=M\cap
E\left(  X_{i}\right)  $ is a maximum matching of $X_{i}$, and $\mu(v_{i}\circ
X_{i})=\mu(X_{i})+1$. If $M_{2}$ is a maximum matching of $v_{i}\circ X_{i}$,
then $\left\vert M_{3}\right\vert =\left\vert M\right\vert $, where
$M_{3}=\left(  M-v_{i}v_{j}-M_{1}\right)  \cup M_{2}$. We continue this
procedure till for every edge in the maximum matching both its end vertices
belong to $F$. Finally, we arrive at a maximum matching $M^{\prime}$ such that
the matching $M^{\prime}\cap E\left(  H\right)  $ saturates vertices from $F$
only. Clearly, $\left\vert M^{\prime}\cap E\left(  H\right)  \right\vert
=\mu\left(  H\left[  F\right]  \right)  $, which completes the proof of Part
\textit{(ii)}.

\textit{(iii)} It goes directly by definition of the set $F$.
\end{proof}

The formulae presented in Theorem \ref{THeorem1} were announced in
\cite{LevMan2024c}. An independent approach to the independence and matching
numbers of corona graphs may be found in \cite{Hong2024}. It worth mentioning
that our proofs are substantially shorter that the ones from \cite{Hong2024}.

It is easy to see that $H\circ\mathcal{X}$ is a bipartite graph if and only if
$H$ is bipartite and every $X_{i}\in$ $\mathcal{X}$ is an edgeless graph. In
particular, $H\circ\mathcal{X}$ is a tree if and only if $H$ is a tree and
every $X_{i}\in$ $\mathcal{X}$ is an edgeless graph.

\begin{theorem}
\label{th7}Let $G=H\circ\mathcal{X}$, where $\mathcal{X}=\left\{  X_{i}:1\leq
i\leq n\left(  H\right)  \right\}  $. Then $G$ is a K\"{o}nig-Egerv\'{a}ry
graph if and only if each $X_{i}$ is a K\"{o}nig-Egerv\'{a}ry graph without a
perfect matching.
\end{theorem}

\begin{proof}
Let $F=\left\{  v_{i}\in V\left(  H\right)  :\mu(X_{i})=\mu(v_{i}\circ
X_{i})\right\}  $.

Notice that $\alpha\left(  X\right)  =\alpha\left(  v\circ X\right)  $, while
$\mu\left(  X\right)  =\mu\left(  v\circ X\right)  $ if and only if $X$ has a
perfect matching.

Now, an equivalent reformulation of the theorem reads as follows: "$G$ is a
K\"{o}nig-Egerv\'{a}ry graph if an only if $F=\emptyset$, and each $X_{i}$ is
a K\"{o}nig-Egerv\'{a}ry graph."

First, let us show that if $G$ is a K\"{o}nig-Egerv\'{a}ry graph, then
$F=\emptyset$. By definition,\textrm{ }$G$ is a K\"{o}nig-Egerv\'{a}ry graph
if and only if $\alpha(G)+\mu(G)=n\left(  G\right)  $.

Since $n\left(  G\right)  =n\left(  H\right)  +$ $%
{\displaystyle\sum\limits_{i=1}^{n\left(  H\right)  }}
n(X_{i})$ and using Theorem \ref{THeorem1}, we infer the following:%
\begin{align*}
\alpha(G)+\mu(G) &  =n\left(  G\right)  \Leftrightarrow\\%
{\displaystyle\sum\limits_{i=1}^{n\left(  H\right)  }}
\alpha(X_{i})+\mu\left(  H\left[  F\right]  \right)  +%
{\displaystyle\sum\limits_{i=1}^{n\left(  H\right)  }}
\mu(v_{i}\circ X_{i}) &  =n\left(  H\right)  +%
{\displaystyle\sum\limits_{i=1}^{n\left(  H\right)  }}
n(X_{i})\Leftrightarrow
\end{align*}%
\[
\left(
{\displaystyle\sum\limits_{i=1}^{n\left(  H\right)  }}
\alpha(X_{i})+%
{\displaystyle\sum\limits_{i=1}^{n\left(  H\right)  }}
\mu(v_{i}\circ X_{i})-%
{\displaystyle\sum\limits_{i=1}^{n\left(  H\right)  }}
n(X_{i})\right)  +\mu\left(  H\left[  F\right]  \right)  =n\left(  H\right)
\Leftrightarrow
\]%
\[
\left(  \left(
{\displaystyle\sum\limits_{i=1}^{n\left(  H\right)  }}
\alpha(X_{i})+%
{\displaystyle\sum\limits_{i=1}^{n\left(  H\right)  }}
\mu(X_{i})-%
{\displaystyle\sum\limits_{i=1}^{n\left(  H\right)  }}
n(X_{i})\right)  +\left\vert V\left(  H\right)  -F\right\vert \right)
+\mu\left(  H\left[  F\right]  \right)  =n\left(  H\right)  \Leftrightarrow
\]%
\[
\left(
{\displaystyle\sum\limits_{i=1}^{n\left(  H\right)  }}
\left(  \alpha(X_{i})+\mu(X_{i})-n(X_{i})\right)  +\left\vert V\left(
H\right)  -F\right\vert \right)  +\mu\left(  H\left[  F\right]  \right)
=n\left(  H\right)  .
\]

Since $\alpha(Y)+\mu(Y)\leq n(Y)$ for every graph $Y$, we obtain
\begin{gather*}
\left\vert V\left(  H\right)  -F\right\vert +\mu\left(  H\left[  F\right]
\right)  \geq n\left(  H\right)  \Leftrightarrow n\left(  H\right)
-\left\vert F\right\vert +\mu\left(  H\left[  F\right]  \right)  \geq n\left(
H\right)  \Leftrightarrow\\
\mu\left(  H\left[  F\right]  \right)  \geq\left\vert F\right\vert .
\end{gather*}

Thus $F=\emptyset$, as required.

Second, let us prove that each $X_{i}$ is a K\"{o}nig-Egerv\'{a}ry graph as
well. Recall the following formula%
\[
\left(
{\displaystyle\sum\limits_{i=1}^{n\left(  H\right)  }}
\left(  \alpha(X_{i})+\mu(X_{i})-n(X_{i})\right)  +\left\vert V\left(
H\right)  -F\left(  H\right)  \right\vert \right)  +\mu\left(  H\left[
F\left(  H\right)  \right]  \right)  =n\left(  H\right)  .
\]

In the context of the fact that $F\left(  H\right)  =\emptyset$, we obtain%
\[%
{\displaystyle\sum\limits_{i=1}^{n\left(  H\right)  }}
\left(  \alpha(X_{i})+\mu(X_{i})-n(X_{i})\right)  =0,
\]
which can be true if and only if each $X_{i}$ is a K\"{o}nig-Egerv\'{a}ry graph.

Conversely, if $F\left(  H\right)  =\emptyset$, and each $X_{i}$ is a
K\"{o}nig-Egerv\'{a}ry graph, then%
\begin{gather*}
\left\vert V\left(  H\right)  -F\left(  H\right)  \right\vert =\left\vert
V\left(  H\right)  \right\vert =n\left(  H\right)  ,\qquad\mu\left(  H\left[
F\left(  H\right)  \right]  \right)  =0,\\%
{\displaystyle\sum\limits_{i=1}^{n\left(  H\right)  }}
\left(  \alpha(X_{i})+\mu(X_{i})-n(X_{i})\right)  =0,
\end{gather*}
and, consequently,
\[
\left(
{\displaystyle\sum\limits_{i=1}^{n\left(  H\right)  }}
\left(  \alpha(X_{i})+\mu(X_{i})-n(X_{i})\right)  +\left\vert V\left(
H\right)  -F\left(  H\right)  \right\vert \right)  +\mu\left(  H\left[
F\left(  H\right)  \right]  \right)  =n\left(  H\right)  ,
\]
that is equivalent to
\[%
{\displaystyle\sum\limits_{i=1}^{n\left(  H\right)  }}
\alpha(X_{i})+\mu\left(  H\left[  F\left(  H\right)  \right]  \right)  +%
{\displaystyle\sum\limits_{i=1}^{n\left(  H\right)  }}
\mu(v_{i}\circ X_{i})=n\left(  H\right)  +%
{\displaystyle\sum\limits_{i=1}^{n\left(  H\right)  }}
n(X_{i}),
\]
i.e., $G$\ is a K\"{o}nig-Egerv\'{a}ry graph.
\end{proof}

\begin{corollary}
$G=H\circ X$ is a K\"{o}nig-Egerv\'{a}ry graph if and only if $X$ is a
K\"{o}nig-Egerv\'{a}ry graph without a perfect matching.
\end{corollary}

Notice that the complete graph $K_{n}$ is a K\"{o}nig-Egerv\'{a}ry graph if
and only if $n\in\left\{  1,2\right\}  $.

\begin{corollary}
The clique corona graph $G=H\circ K_{n}$ is a K\"{o}nig-Egerv\'{a}ry graph if
and only if $n=1$.
\end{corollary}

For instance, $G=P_{2}\circ\{C_{4},K_{1}\}$ is not a K\"{o}nig-Egerv\'{a}ry
graph, as $C_{4}$ is a K\"{o}nig-Egerv\'{a}ry graph with perfect matchings.
Another example is $G=P_{3}\circ\{K_{3},K_{1},K_{1}\}$, which is not a
K\"{o}nig-Egerv\'{a}ry graph since $K_{3}$ is not a K\"{o}nig-Egerv\'{a}ry
graph. Even more general, $K_{1}\circ\{qK_{1}\cup pK_{2}\}$ is a
K\"{o}nig-Egerv\'{a}ry graph, for every $q\geq1$ and $p\geq0$, while the graph
$K_{1}\circ\{pK_{2}\}$ is not a K\"{o}nig-Egerv\'{a}ry graph, but is a
$1$-K\"{o}nig-Egerv\'{a}ry graph, for every $p\geq1$.

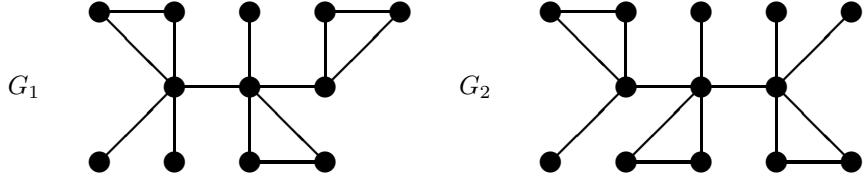
\begin{figure}[h]
\setlength{\unitlength}{1cm}\begin{picture}(5,2.6)\thicklines
\multiput(2,0)(1,0){4}{\circle*{0.29}}
\multiput(3,1)(1,0){3}{\circle*{0.29}}
\multiput(2,2)(1,0){5}{\circle*{0.29}}
\put(3,1){\line(1,0){2}}
\put(3,0){\line(0,1){2}}
\put(2,0){\line(1,1){1}}
\put(2,2){\line(1,-1){1}}
\put(2,2){\line(1,0){1}}
\put(4,0){\line(0,1){2}}
\put(4,0){\line(1,0){1}}
\put(4,1){\line(1,-1){1}}
\put(5,1){\line(0,1){1}}
\put(5,1){\line(1,1){1}}
\put(5,2){\line(1,0){1}}
\put(1,1){\makebox(0,0){$G_{1}$}}
\multiput(8,0)(1,0){5}{\circle*{0.29}}
\multiput(9,1)(1,0){3}{\circle*{0.29}}
\multiput(8,2)(1,0){5}{\circle*{0.29}}
\put(8,0){\line(1,1){1}}
\put(8,2){\line(1,0){1}}
\put(8,2){\line(1,-1){1}}
\put(9,0){\line(1,1){1}}
\put(9,1){\line(1,0){2}}
\put(9,1){\line(0,1){1}}
\put(9,0){\line(1,0){1}}
\put(10,0){\line(0,1){2}}
\put(11,0){\line(0,1){2}}
\put(11,0){\line(1,0){1}}
\put(11,1){\line(1,1){1}}
\put(11,1){\line(1,-1){1}}
\put(7,1){\makebox(0,0){$G_{2}$}}
\end{picture}\caption{$G_{1}$ is an $1$-K\"{o}nig-Egerv\'{a}ry graph, while
$G_{2}$ is a K\"{o}nig-Egerv\'{a}ry graph.}%
\label{fig4}%
\end{figure}

\begin{corollary}
Let $G=H\circ\mathcal{X}$, where $\mathcal{X}=\left\{  X_{i}:1\leq i\leq
n\left(  H\right)  \right\}  $. Then $G$ is a K\"{o}nig-Egerv\'{a}ry graph
with a perfect matching if and only if each $X_{i}$ is a
K\"{o}nig-Egerv\'{a}ry graph with an almost perfect matching.
\end{corollary}

\begin{proof}
According to Theorem \ref{th7}, each $X_{i}$ is a K\"{o}nig-Egerv\'{a}ry graph
without a perfect matching. Assume, to the contrary, that some $X_{j}$ has no
almost perfect matchings. Then at least two vertices of $X_{j}$ remain
non-saturated by every maximum matching of $X_{j}$. Consequently, no maximum
matching of $G$ can saturates all the vertices of $X_{j}$, in contradiction
with the fact that $G$ has a perfect matching.

Conversely, by Theorem \ref{th7}, $G$ is a K\"{o}nig-Egerv\'{a}ry graph. Let
$M_{i}$ denote an almost perfect matching of $X_{i}$ and $x_{i}\in V\left(
X_{i}\right)  $ be the non-saturated vertex of $X_{i}$. Then $\left\{
v_{i}x_{i}:1\leq i\leq n\left(  H\right)  \right\}  \cup M_{1}\cup M_{2}%
\cup\cdots\cup M_{n(H)}$ is a perfect matching of $G$.
\end{proof}

Clearly, a graph has a unique almost perfect matching if and only if it is a
disjoint union of a graph with a unique perfect matching and $K_{1}$.

\begin{corollary}
The graph $H\circ\left\{  X_{i}:1\leq i\leq n\left(  H\right)  \right\}  $ is
a K\"{o}nig-Egerv\'{a}ry graph with a unique perfect matching if and only if
each $X_{i}$ is a K\"{o}nig-Egerv\'{a}ry graph with a unique almost perfect matching.
\end{corollary}

\begin{corollary}
The graph $H\circ X$ is a K\"{o}nig-Egerv\'{a}ry graph with a perfect matching
if and only if $X$ is a K\"{o}nig-Egerv\'{a}ry graph with an almost perfect
matching. Moreover, $H\circ X$ has a unique perfect matching if and only if
$X$ is a K\"{o}nig-Egerv\'{a}ry graph with a unique almost perfect matching.
\end{corollary}

\begin{corollary}
The clique corona graph $H\circ K_{n}$ is a K\"{o}nig-Egerv\'{a}ry graph with
a perfect matching if and only if $n=1$.
\end{corollary}

\begin{lemma}
\label{lem3}For every graph $G$ the inequalities $\frac{n\left(  G\right)
}{2}\geq\mu\left(  G\right)  \geq n\left(  G\right)  -1$ hold if and only if
$G=K_{0}$, $G=K_{1}$ or $G=K_{2}$.
\end{lemma}

\begin{proof}
Since $\frac{n\left(  G\right)  }{2}\geq\mu\left(  G\right)  \geq n\left(
G\right)  -1$, we obtain that $n\left(  G\right)  \leq2$.

The case $n\left(  G\right)  =0$ is suitable, because $n(G)=0\geq\mu\left(
G\right)  =0\geq-1=n\left(  G\right)  -1$ is true, i.e. $G=K_{0}$.

On the other hand, $n\left(  G\right)  =1$ means that $G=K_{1}$, while
$n\left(  G\right)  =2$ \ implies $G=K_{2}$, because $\mu\left(
2K_{1}\right)  =0>n\left(  2K_{1}\right)  -1=1$ is not true.
\end{proof}

For $k\geq3$, consider the following graphs: $G_{1}=P_{k}\circ\left\{
K_{2},K_{1},...,K_{1}\right\}  ,G_{2}=P_{k}\circ\left\{  K_{2},K_{2}%
,K_{1},...,K_{1}\right\}  $ and $G_{3}=P_{k}\circ\left\{  K_{2},K_{1}%
,K_{2},K_{1},K_{1},...,K_{1}\right\}  $. Notice that $G_{1}$ has $F=K_{1}$ and
is $1$-K\"{o}nig-Egerv\'{a}ry, $G_{2}$ has $F=K_{2}$ and is $1$%
-K\"{o}nig-Egerv\'{a}ry, while $G_{3}$ has $F=2K_{1}$ and it is not $1$-K\"{o}nig-Egerv\'{a}ry.

\begin{theorem}
\label{th8}Let $G=H\circ\mathcal{X}$, where $\mathcal{X}=\left\{  X_{i}:1\leq
i\leq n\left(  H\right)  \right\}  $, and
\[
F=\left\{  v_{i}\in V\left(  H\right)  :\mu(X_{i})=\mu(v_{i}\circ
X_{i})\right\}  .
\]
Then $G$ is a $1$-K\"{o}nig-Egerv\'{a}ry graph if and only if

(i) $H\left[  F\right]  =K_{0}$ and one of $X_{i}$ is $1$%
-K\"{o}nig-Egerv\'{a}ry while the others are K\"{o}nig-Egerv\'{a}ry;

or

(ii) $H\left[  F\right]  =K_{1}$ and all $X_{i}$ are K\"{o}nig-Egerv\'{a}ry graphs;

or

(iii) $H\left[  F\right]  =K_{2}$ and all $X_{i}$ are K\"{o}nig-Egerv\'{a}ry graphs.
\end{theorem}

\begin{proof}
Firstly, since $n\left(  G\right)  =n\left(  H\right)  +$ $%
{\displaystyle\sum\limits_{i=1}^{n\left(  H\right)  }}
n(X_{i})$ and using Theorem \ref{THeorem1}, we infer the following:%
\begin{align*}
\alpha(G)+\mu(G)  &  =n\left(  G\right)  -1\Leftrightarrow\\%
{\displaystyle\sum\limits_{i=1}^{n\left(  H\right)  }}
\alpha(X_{i})+\mu\left(  H\left[  F\right]  \right)  +%
{\displaystyle\sum\limits_{i=1}^{n\left(  H\right)  }}
\mu(v_{i}\circ X_{i})  &  =n\left(  H\right)  +%
{\displaystyle\sum\limits_{i=1}^{n\left(  H\right)  }}
n(X_{i})-1\Leftrightarrow
\end{align*}%
\[
\left(
{\displaystyle\sum\limits_{i=1}^{n\left(  H\right)  }}
\alpha(X_{i})+%
{\displaystyle\sum\limits_{i=1}^{n\left(  H\right)  }}
\mu(v_{i}\circ X_{i})-%
{\displaystyle\sum\limits_{i=1}^{n\left(  H\right)  }}
n(X_{i})\right)  +\mu\left(  H\left[  F\right]  \right)  =n\left(  H\right)
-1\Leftrightarrow
\]%
\[
\left(  \left(
{\displaystyle\sum\limits_{i=1}^{n\left(  H\right)  }}
\alpha(X_{i})+%
{\displaystyle\sum\limits_{i=1}^{n\left(  H\right)  }}
\mu(X_{i})-%
{\displaystyle\sum\limits_{i=1}^{n\left(  H\right)  }}
n(X_{i})\right)  +\left\vert V\left(  H\right)  -F\right\vert \right)
+\mu\left(  H\left[  F\right]  \right)  =n\left(  H\right)  -1\Leftrightarrow
\]%
\[%
{\displaystyle\sum\limits_{i=1}^{n\left(  H\right)  }}
\left(  \alpha(X_{i})+\mu(X_{i})-n(X_{i})\right)  +\left\vert V\left(
H\right)  -F\right\vert +\mu\left(  H\left[  F\right]  \right)  =n\left(
H\right)  -1.
\]

Since $\alpha(Y)+\mu(Y)\leq n(Y)$ for every graph $Y$, we obtain
\begin{gather*}
\left\vert V\left(  H\right)  -F\right\vert +\mu\left(  H\left[  F\right]
\right)  \geq n\left(  H\right)  -1\Leftrightarrow\\
n\left(  H\right)  -\left\vert F\right\vert +\mu\left(  H\left[  F\right]
\right)  \geq n\left(  H\right)  -1\Leftrightarrow\mu\left(  H\left[
F\right]  \right)  \geq\left\vert F\right\vert -1.
\end{gather*}

By Lemma \ref{lem3}, we infer that $H\left[  F\right]  =K_{0}$ or $H\left[
F\right]  =K_{1}$ or $H\left[  F\right]  =K_{2}$.

Secondly, let us prove that each $X_{i}$ is a K\"{o}nig-Egerv\'{a}ry graph.
For this purpose, recall the following formula:%
\begin{equation}%
{\displaystyle\sum\limits_{i=1}^{n\left(  H\right)  }}
\left(  \alpha(X_{i})+\mu(X_{i})-n(X_{i})\right)  +\left\vert V\left(
H\right)  -F\right\vert +\mu\left(  H\left[  F\right]  \right)  =n\left(
H\right)  -1. \tag{*}%
\end{equation}

\textit{Case 0.} $H\left[  F\right]  =K_{0}$. By the equality (*), we obtain
\[%
{\displaystyle\sum\limits_{i=1}^{n\left(  H\right)  }}
\left(  \alpha(X_{i})+\mu(X_{i})-n(X_{i})\right)  +n(H)-0+0=n\left(  H\right)
-1.
\]

Hence,%
\[%
{\displaystyle\sum\limits_{i=1}^{n\left(  H\right)  }}
\left(  \alpha(X_{i})+\mu(X_{i})-n(X_{i})\right)  =-1,
\]
which can be true if and only if $\alpha(X_{i})+\mu(X_{i})-n(X_{i})=0$ for
every $i\in\left\{  1,\cdots,n\left(  H\right)  \right\}  $ but one, i.e., one
of $X_{i}$ is $1$-K\"{o}nig-Egerv\'{a}ry while the others are K\"{o}nig-Egerv\'{a}ry.

\textit{Case 1.} $H\left[  F\right]  =K_{1}$. By the equality (*), we obtain
\[%
{\displaystyle\sum\limits_{i=1}^{n\left(  H\right)  }}
\left(  \alpha(X_{i})+\mu(X_{i})-n(X_{i})\right)  +n(H)-1+0=n\left(  H\right)
-1.
\]

Hence,%
\[%
{\displaystyle\sum\limits_{i=1}^{n\left(  H\right)  }}
\left(  \alpha(X_{i})+\mu(X_{i})-n(X_{i})\right)  =0,
\]
which can be true if and only if $\alpha(X_{i})+\mu(X_{i})-n(X_{i})=0$ for
every $i\in\left\{  1,\cdots,n\left(  H\right)  \right\}  $, i.e., each
$X_{i}$ is a K\"{o}nig-Egerv\'{a}ry graph.

\textit{Case 2.} $H\left[  F\right]  =K_{2}$. By the equality (*), we obtain%
\[%
{\displaystyle\sum\limits_{i=1}^{n\left(  H\right)  }}
\left(  \alpha(X_{i})+\mu(X_{i})-n(X_{i})\right)  +n(H)-2+1=n\left(  H\right)
-1.
\]
Hence, $%
{\displaystyle\sum\limits_{i=1}^{n\left(  H\right)  }}
\left(  \alpha(X_{i})+\mu(X_{i})-n(X_{i})\right)  =0$, which is true if and
only if each $X_{i}$ is K\"{o}nig-Egerv\'{a}ry, as in \textit{Case 1}.

In conclusion, if $G$ is a $1$-K\"{o}nig-Egerv\'{a}ry graph, then $H\left[
F\right]  =K_{1}$ or $H\left[  F\right]  =K_{2}$, and all $X_{i}$ are
K\"{o}nig-Egerv\'{a}ry graphs.

Conversely, if $H\left[  F\right]  =K_{0}$, one $X_{i}$, say $X_{j}$, is
$1$-K\"{o}nig-Egerv\'{a}ry, while the others are K\"{o}nig-Egerv\'{a}ry, then%
\begin{gather*}
\left\vert V\left(  H\right)  -F\right\vert =\left\vert V\left(  H\right)
\right\vert =n\left(  H\right)  ,\mu\left(  H\left[  F\right]  \right)  =0,\\%
{\displaystyle\sum\limits_{i=1}^{n\left(  H\right)  }}
\left(  \alpha(X_{i})+\mu(X_{i})-n(X_{i})\right)  =-1,
\end{gather*}
and, consequently,
\[%
{\displaystyle\sum\limits_{i=1}^{n\left(  H\right)  }}
\left(  \alpha(X_{i})+\mu(X_{i})-n(X_{i})\right)  +\left\vert V\left(
H\right)  -F\right\vert +\mu\left(  H\left[  F\right]  \right)  =n\left(
H\right)  -1,
\]
that, by Theorem \ref{THeorem1}\textit{(iii)}, is equivalent to the equality
\[%
{\displaystyle\sum\limits_{i=1}^{n\left(  H\right)  }}
\alpha(X_{i})+\mu\left(  H\left[  F\right]  \right)  +%
{\displaystyle\sum\limits_{i=1}^{n\left(  H\right)  }}
\mu(v_{i}\circ X_{i})=n\left(  H\right)  -1+%
{\displaystyle\sum\limits_{i=1}^{n\left(  H\right)  }}
n(X_{i}),
\]
which, by Theorem \ref{THeorem1}\textit{(i)} and Theorem \ref{THeorem1}%
\textit{(ii)}, means that%
\[
\alpha(G)+\mu\left(  G\right)  =n\left(  G\right)  -1,
\]
i.e., $G$\ is a $1$-K\"{o}nig-Egerv\'{a}ry graph.

Further, if $H\left[  F\right]  =K_{1}$, and each $X_{i}$ is a
K\"{o}nig-Egerv\'{a}ry graph, then%
\begin{gather*}
\left\vert V\left(  H\right)  -F\right\vert =\left\vert V\left(  H\right)
\right\vert -1=n\left(  H\right)  -1,\mu\left(  H\left[  F\right]  \right)
=0,\\%
{\displaystyle\sum\limits_{i=1}^{n\left(  H\right)  }}
\left(  \alpha(X_{i})+\mu(X_{i})-n(X_{i})\right)  =0,
\end{gather*}
and, consequently,
\[%
{\displaystyle\sum\limits_{i=1}^{n\left(  H\right)  }}
\left(  \alpha(X_{i})+\mu(X_{i})-n(X_{i})\right)  +\left\vert V\left(
H\right)  -F\right\vert +\mu\left(  H\left[  F\right]  \right)  =n\left(
H\right)  -1,
\]
that, by Theorem \ref{THeorem1}\textit{(iii)}, is equivalent to the equality
\[%
{\displaystyle\sum\limits_{i=1}^{n\left(  H\right)  }}
\alpha(X_{i})+\mu\left(  H\left[  F\right]  \right)  +%
{\displaystyle\sum\limits_{i=1}^{n\left(  H\right)  }}
\mu(v_{i}\circ X_{i})=n\left(  H\right)  -1+%
{\displaystyle\sum\limits_{i=1}^{n\left(  H\right)  }}
n(X_{i}),
\]
which, by Theorem \ref{THeorem1}\textit{(i)} and Theorem \ref{THeorem1}%
\textit{(ii)}, means that%
\[
\alpha(G)+\mu\left(  G\right)  =n\left(  G\right)  -1,
\]
i.e., $G$\ is a $1$-K\"{o}nig-Egerv\'{a}ry graph.

Finally, if $H\left[  F\right]  =K_{2}$, and each $X_{i}$ is a
K\"{o}nig-Egerv\'{a}ry graph, then%
\begin{gather*}
\left\vert V\left(  H\right)  -F\right\vert =\left\vert V\left(  H\right)
\right\vert -2=n\left(  H\right)  -2,\mu\left(  H\left[  F\right]  \right)
=1,\\%
{\displaystyle\sum\limits_{i=1}^{n\left(  H\right)  }}
\left(  \alpha(X_{i})+\mu(X_{i})-n(X_{i})\right)  =0,
\end{gather*}
and, therefore,
\[%
{\displaystyle\sum\limits_{i=1}^{n\left(  H\right)  }}
\left(  \alpha(X_{i})+\mu(X_{i})-n(X_{i})\right)  +\left\vert V\left(
H\right)  -F\right\vert +\mu\left(  H\left[  F\right]  \right)  =n\left(
H\right)  -1,
\]
that, by Theorem \ref{THeorem1}\textit{(iii)}, is equivalent to the equality
\[%
{\displaystyle\sum\limits_{i=1}^{n\left(  H\right)  }}
\alpha(X_{i})+\mu\left(  H\left[  F\right]  \right)  +%
{\displaystyle\sum\limits_{i=1}^{n\left(  H\right)  }}
\mu(v_{i}\circ X_{i})=n\left(  H\right)  -1+%
{\displaystyle\sum\limits_{i=1}^{n\left(  H\right)  }}
n(X_{i}),
\]
which, by Theorem \ref{THeorem1}\textit{(i)} and Theorem \ref{THeorem1}%
\textit{(ii)}, means that%
\[
\alpha(G)+\mu\left(  G\right)  =n\left(  G\right)  -1,
\]
i.e., $G$\ is a $1$-K\"{o}nig-Egerv\'{a}ry graph.
\end{proof}

For instance, the graph $G=K_{2}\circ\left\{  K_{4},K_{2}\right\}  $ has
$K_{2}\left[  F\right]  =K_{2}$, but is not a $1$-K\"{o}nig-Egerv\'{a}ry
graph, because $K_{4}$ is not K\"{o}nig-Egerv\'{a}ry.

\begin{corollary}
$G=H\circ X$ is a $1$-K\"{o}nig-Egerv\'{a}ry graph if and only if

\textit{(i)} $G=K_{1}\circ X$, where either $X$ is $1$-K\"{o}nig-Egerv\'{a}ry
without a perfect matching or $X$ is K\"{o}nig-Egerv\'{a}ry with a perfect matching;

or

(ii) $G=K_{2}\circ X$ and $X$ is K\"{o}nig-Egerv\'{a}ry with a perfect matching.
\end{corollary}

\begin{proof}
By Theorem \ref{th8}, if $n\left(  H\right)  \geq3$, then $G$ is not $1$-K\"{o}nig-Egerv\'{a}ry.

Assume that $G$ is a $1$-K\"{o}nig-Egerv\'{a}ry graph.

If $n\left(  H\right)  =1$, then either $X=K_{3}$ or, according to Theorem
\ref{th8}, $X$ is K\"{o}nig-Egerv\'{a}ry with a perfect matching.

If $n\left(  H\right)  =2$, then $X$ is a K\"{o}nig-Egerv\'{a}ry graph with a
perfect matching, by Theorem \ref{th8}.

The converse follows from Theorem \ref{th8} and the fact that, clearly,
$G=K_{1}\circ K_{3}$ is a $1$-K\"{o}nig-Egerv\'{a}ry graph.
\end{proof}

\begin{example}
The following graphs are $1$-K\"{o}nig-Egerv\'{a}ry: $K_{2}\circ\left\{
K_{3},K_{1}\right\}  $, $P_{3}\circ\left\{  K_{3},K_{1},K_{1}\right\}  $,
$P_{3}\circ\left\{  K_{1},K_{3},K_{1}\right\}  $.
\end{example}

Notice that $K_{n}$ is a $1$-K\"{o}nig-Egerv\'{a}ry graph if and only if
$n\in\left\{  3,4\right\}  $.

\begin{corollary}
The clique corona graph $G=H\circ K_{n}$ is $1$-K\"{o}nig-Egerv\'{a}ry if and
only if $G\in\left\{  K_{1}\circ K_{2},K_{1}\circ K_{3},K_{1}\circ K_{4}%
,K_{2}\circ K_{2}\right\}  $.
\end{corollary}

\section{Conclusions}

The graphs $G_{1}=P_{4}\circ K_{2}$ and $G_{2}=P_{4}\circ\left\{  K_{2}%
,K_{2},K_{2},P_{3}\right\}  $ are $2$-K\"{o}nig-Egerv\'{a}ry graphs, since
$n\left(  G_{1}\right)  =12,\alpha\left(  G_{1}\right)  =4$ and $\mu\left(
G_{1}\right)  =6$, while $n\left(  G_{2}\right)  =13,\alpha\left(
G_{2}\right)  =5$ and $\mu\left(  G_{2}\right)  =6$.

Let notice that:

\begin{itemize}
\item $K_{5}=K_{1}\circ K_{4}$ and $K_{6}=K_{1}\circ K_{5}$ are $2$%
-K\"{o}nig-Egerv\'{a}ry graphs;

\item $K_{7}=K_{1}\circ K_{7}$ and $K_{8}=K_{1}\circ K_{7}$ are $3$%
-K\"{o}nig-Egerv\'{a}ry graphs;

\item $K_{9}=K_{1}\circ K_{8}$ and $K_{10}=K_{1}\circ K_{9}$ are
$4$-K\"{o}nig-Egerv\'{a}ry graphs.
\end{itemize}

It motivates the following.

\begin{problem}
Characterize (corona) graphs that are $k$-K\"{o}nig-Egerv\'{a}ry for $k\geq2$.
\end{problem}

\end{document}